\documentclass[12pt]{amsart}
\usepackage{amssymb, amsfonts}

\renewcommand{\a}{\alpha}
\renewcommand{\o}{\omega}

\renewcommand{\d}{\delta}

\newcommand{\e}{\varepsilon}
\newcommand{\g}{\gamma}

\newcommand{\s}{\sigma}

\newcommand{\tr}{\mathop{\mathrm{tr}}\nolimits}

\newcommand{\cH}{{\mathcal H}}

\newcommand{\cL}{{\mathcal L}}

\newcommand{\bZ}{{\mathbb Z}}

\newcommand{\fG}{{\mathfrak g}}

\newtheorem{theorem}{Theorem}[section]

\newtheorem{lemma}[theorem]{Lemma} 

\theoremstyle{definition}   

\begin{document}

\title[Distances between Matrix Algebras]{Distances between
matrix algebras \\ that converge to coadjoint orbits }
\author{Marc A. Rieffel}
\address{Department of Mathematics \\
University of California \\ Berkeley, CA 94720-3840}
\email{rieffel@math.berkeley.edu}
\thanks{The research reported here was
supported in part by National Science Foundation grant DMS-0753228.}

\begin{abstract}
For any sequence of matrix algebras that converge to a coadjoint
orbit we give explicit formulas that show that the distances between
the matrix algebras (viewed as quantum metric spaces) converges
to 0. In the process we develop a general point of view that is 
likely to be useful in other related settings.
\end{abstract}

\maketitle
\allowdisplaybreaks

\section*{Introduction}

In earlier papers \cite{R6, R7, R21} I provided ways to give a precise meaning
to statements in the literature of high-energy physics and string
theory of the kind ``Matrix algebras converge to the sphere''.  
I did this by equipping the matrix algebras with suitable ``Lipschitz
seminorms'' that make the matrix algebras into
``compact quantum metric spaces'', and then by defining convergence
by means of a suitable ``quantum Gromov-Hausdorff distance''
between quantum metric spaces. By now a number of variations on
this approach have been studied \cite{Krr, KrL, Lih2, Lih3, Lih4, Wuw3}.

When I then began to examine what consequences the
convergence of quantum metric spaces had for the convergence
of ``vector bundles'' (i.e. projective modules) over them  \cite{R17}, 
I found that it
is very important that the Lipschitz seminorms satisfy a suitable 
Leibniz property. In \cite{R21} I showed that a very convenient source 
for seminorms that satisfy this Leibniz property consisted of normed 
bimodules, and in \cite{R21} I also constructed explicit normed bimodules that
worked well for matrix algebras converging to coadjoint orbits.

However, for our approach to work well, it should be the case that 
for a convergent sequence of matrix algebras the quantum 
Gromov-Hausdorff distances between the matrix algebras
go to 0; but when I required that all of the seminorms satisfy the 
Leibniz property I did not see at first how to show this convergence
directly. The purpose of the present paper is to give explicit
normed bimodules and corresponding Leibniz Lipschitz seminorms 
that demonstrate this convergence to 0. In the process we
develop a general point of view that is likely to be useful in other related
situations. This point of view is motivated by the ``nuclear distance''
introduced and studied by Hanfeng Li \cite{KrL, Lih3, Lih4}, in
which all of the bimodules are required to be $C^*$-algebras. I
have so far not seen how to apply Hanfeng Li's approach directly 
to obtain explicit normed bimodules for the matrix-algebra case.
But by trying to arrange that all of the normed bimodules 
that I used were
$C^*$-algebras I was led to see the path to the explicit
bimodules that I sought.

The first section of this paper recalls the setting for matrix
algebras converging to coadjoint orbits, reformulates the
bimodules from \cite{R21} so that they are $C^*$-algebras, and
then uses these reformulated bimodules to construct candidates for
$C^*$-bimodules between matrix algebras whose Leibniz Lipschitz
seminorms might show that the distances go to 0. In Section 2
we place matters in a general framework, and obtain a basic theorem
in this general framework.   
In Section 3 we prove that the candidate bimodules and corresponding
Lipschitz seminorms of Section 1 do indeed show that the distances between
the converging matrix algebras go to 0. An important step in the proof comes 
from the general theorem in Section 2. The full statement of the main theorem
is given at the end of Section 3.


\section{The bimodules}
\label{sec1}

We recall the setting from \cite{R7, R21}.
We let $G$ be a compact connected
semisimple Lie group, and we let $\fG$
denotes the complexification of the Lie algebra of $G$.  We choose a maximal torus in
$G$, with corresponding Cartan subalgebra of $\fG$, its set of roots,
and a choice of positive roots.  We fix a specific irreducible
unitary representation, $(U,\cH)$, of $G$, and we choose a highest-weight
vector, $\xi$, for $(U,\cH)$ with $\|\xi\| = 1$.  For any $n \in \bZ_{\ge 1}$ we set
$\xi^n = \xi^{\otimes n}$ in $\cH^{\otimes n}$, and we let
$(U^n,\cH^n)$ be the restriction of $U^{\otimes n}$ to the 
$U^{\otimes n}$-invariant subspace, $\cH^n$, of $\cH^{\otimes n}$ that is
generated by $\xi^n$.  Then $(U^n,\cH^n)$ is an irreducible
representation of $G$ with highest-weight vector $\xi^n$, and its
highest weight is just $n$ times the highest weight of $(U,\cH)$.  We
denote the dimension of $\cH^n$ by $d_n$.

We let $B^n = \cL(\cH^n)$.  The action of $G$ on $B^n$ by conjugation
by $U^n$ will be denoted simply by $\a$.  We assume that a continuous
length function, $\ell$, has been chosen for $G$, and we denote the
corresponding $C^*$-metric on $B^n$ by $L_n^B$.  It is defined by
\[
L_n^B(T) = \sup\{\|\a_x(T)-T\|/\ell(x): x \notin e_G\}
\]
for $T \in B^n$. (The term ``$C^*$-metric'' is defined in definition 4.1
of \cite{R21}.)
We let $P^n$ denote
the rank-one projection along $\xi^n$.  Then the $\a$-stability
subgroup, $H$, for $P = P^1$ will also be the $\a$-stability subgroup for
each $P^n$. We let $A = C(G/H)$, and we let $L_A$ be the $C^*$-metric
on $A$ for $\ell$ and the left-translation action of $G$ on $G/H$,
defined much as is $L_n^B$.  

Roughly speaking, our goal is to obtain estimates on the distance
between $(B^m, L_m^B)$ and $(B^n, L_n^B)$ that show that the
distance goes to 0 as $m$ and $n$ go to 0. We want to do this in the
setting of \cite{R21}, where we insist that the Lipschitz seminorms involved 
satisfy a strong Leibniz property. We require this because of its importance 
for treating vector bundles (and projective modules), as shown in \cite{R17}.

But in contrast to \cite{R21}, our presentation here is influenced by
Hanfeng Li's definition of the ``nuclear distance'' between 
quantum metric spaces, 
although I have not seen how to use his nuclear distance directly. The effect
of this influence is that we try to arrange that all of the bimodules that we 
consider are actually $C^*$-algebras.

To motivate the construction of our bimodules, we first reformulate the
corresponding constructions from \cite{R21} in terms of $C^*$-algebras. 
For any given $n$ we form the $C^*$-algebra $A\otimes B^n = 
C(G/H, B^n)$. There are canonical injections of $A$ and $B^n$ into
$A\otimes B^n$, and by means of these we view $A\otimes B^n$
as an $A$-$B^n$-bimodule. Let $\o_n \in C(G/H, B^n)$ be defined by
\[
\o_n(x) = \a_x(P^n).
\]
We use the distinguished element $\o_n$ and the bimodule structure to
define a seminorm, $N_n$, on $A\oplus B^n$ by
\[
N_n(f, T) = \|f\o_n - \o_n T\|.
\]
This seminorm is easily seen to be the same as the seminorm $N_\s$
described by other means in proposition 7.2 of \cite{R21}. It is also easy 
to see that $N_n$ satisfies the strong Leibniz property defined in 
definition 1.1 of \cite{R21}, for the reasons discussed in example 2.3
of \cite{R21} if $A\otimes B^n$ is viewed as an 
$(A\otimes B^n)$-bimodule in the evident way.

For a suitable choice of the constant $\g$, as discussed in 
propositions 8.1 and 8.2 of \cite{R21}, $\g^{-1}N_n$ is a bridge, as 
defined in definition 5.1 of \cite{R6}. This implies that the $*$-seminorm
$L_n$ on $A\oplus B^n$ defined by
\[
L_n(f, T) = L_A(f) \vee L^B_n(T) \vee \g^{-1}(N_n(f, T) \vee N_n(\bar f, T^*))
\]
is a $C^*$-metric on $A\oplus B^n$ (where $\vee$ means ``maximum of'') 
that has the further property that its quotients 
on $A$ and $B^n$ agree with $L^A$ and $L^B_n$ on self-adjoint
elements. (See notation 5.5 and definition 6.1 of \cite{R21}.) 
This quotient condition
on seminorms is exactly what we required in \cite{R6, R7, R21} in order to 
define distances between $C^*$-algebras such as $A$ and $B^n$.
Specifically, for our situation, let $S(A)$ denote the state space of $A$,
and similarly for $B^n$ and $A\oplus B^n$. Then $S(A)$ and $S(B^n)$
are naturally viewed as subsets of $S(A\oplus B^n)$. Now $L_n$
defines a metric, $\rho_{L_n}$, on $S(A\oplus B^n)$ by
\[
\rho_{L_n}(\mu, \nu) = \sup\{|\mu(f, T) - \nu(f, T)|:L_n(f, T) \leq 1\} .
\]
(By definition this supremum should be taken over just self-adjoint 
$f$ and $T$, but by the comments made just before definition 2.1 of \cite{R6} it 
can equivalently be
taken over all $f$ and $T$ because $L_n$ is 
a $*$-seminorm. This fact is also used later for other $*$-seminorms.)
The corresponding ordinary Hausdorff distance
\[
\mathrm{dist}_H^{\rho_{L_n}}(S(A), S(B^n))
\]
gives, by definition, an upper bound for $\mathrm{dist}_q(A, B^n)$
as defined in definition 4.2 of \cite{R6} when we don't require the
strong Leibniz condition, and for $\mathrm{prox}(A, B^n)$ as
defined in definition 5.6 of \cite{R21} when we do require the
strong Leibniz condition. It is shown in theorem 4.3 of \cite{R6}
that $\mathrm{dist}_q$ satisfies the triangle inequality. But
$\mathrm{prox}$ probably does not satisfy the triangle inequality, basically because the quotient of a seminorm that satisfies the Leibniz condition 
need not satisfy the Leibniz condition. We always have 
$\mathrm{dist}_q(A, B) \leq \mathrm{prox}(A, B)$, so if we can show that
$\mathrm{prox}(A, B)$ is ``small'' than it follows that $\mathrm {dist}_q(A, B)$
is ``small'' too.

Since for our specific situation  $\mathrm{prox}(A, B^n)$ converges to 0 
as $n$ goes to $\infty$, as seen in theorem 9.1 of \cite{R21}, (and
similarly for its matricial version, $\mathrm{prox}_s$, by theorem 14.1
of \cite{R21}), it is natural to expect that $\mathrm{prox}(B^m, B^n)$
converges to 0 as $m$ and $n$ go to $\infty$. But because we can
not invoke the triangle inequality, we need to give a direct proof of this fact. 
In the process of doing this
we will construct a specific seminorm that gives quantitative
estimates.

Towards our goal we seek to construct a suitable $B^m$-$B^n$-bimodule.
We can, of course, view the $C^*$-algebra $A\otimes B^m$ as being
the $B^m$-$A$-bimodule $B^m\otimes A$, and then it is natural to
form an ``amalgamation'' over $A$ of these two $C^*$-algebras, to
obtain the $B^m$-$B^n$-bimodule
\[
(B^m\otimes A)\otimes_A(A\otimes B^n) = B^m\otimes A\otimes B^n ,
\]
which we can view as $C(G/H, B^m\otimes B^n)$. Notice that this
is again a $C^*$-algebra, and that we have natural injections of $B^m$ and 
$B^n$ into it. Inside this bimodule we choose a distinguished element,
namely $\o_{mn} = \o_m\otimes \o_n$, viewed as defined by
\[
\o_{mn}(x) = \a_x(P^m)\otimes \a_x(P^n) = \a_x(P^m\otimes P^n).
\]
In terms of $\o_{mn}$ we define a seminorm, $N_{mn}$, on 
$B^m \oplus B^n$ by
\[
N_{mn}(S, T) = \|S\o_{mn} - \o_{mn}T\|,
\]
where the norm is that of the $C^*$-algebra $C(G/H, B^m\otimes B^n)$.
We can now hope to find constants $\g$ such that $\g^{-1}N_{mn}$ is 
a bridge between $B^m$ and $B^n$. In the 
next section we describe a more general setting within which to choose
such bridge constants.


\section{the bridge constants}
\label{sec2}
In this section we consider the following more general setting. We
are given three compact $C^*$-metric spaces, $(A, L_A)$, $(B, L_B)$
and $(C, L_C)$. We are also given unital $C^*$-algebras $D$ and $E$ together
with injective unital homomorphisms of $A$ and $B$ into $D$, and of $B$ and $C$ into $E$.
(Actually, we do not need the unital homomorphisms to be injective, but then we
should provide notation for them, and that would clutter our calculations.) 
Thus we can consider $D$ to be an $A$-$B$-bimodule and $E$ to be a
$B$-$C$-bimodule. We assume further that we are given distinguished elements
$d_0$ and $e_0$ of $D$ and $E$ respectively. For convenience we assume
that $\|d_0\| = 1 = \|e_0\|$. We then define seminorms $N_D$
and $N_E$ on $A\oplus B$ and $B\oplus C$ by
\[
N_D(a,b) = \|ad_0 - d_0b\|_D 
\]
and similarly for $N_E$. We assume that there are constants $\g_D$ and $\g_E$
such that $\g_D^{-1}N_D$ and $\g_E^{-1}N_E$ are bridges for $(L_A, L_B)$
and $(L_B, L_C)$ respectively. This means that when we form the $*$-seminorm
\[
L_{AB}(a,b) = L_A(a)\vee L_B(b)\vee \g_D^{-1}(N_D(a, b)\vee N_D(a^*, b^*), 
\]
its quotients on $A$ and $B$ agree with
$L_A$ and $L_B$ on self-adjoint elements, 
and similarly for $L_{BC}$. Note that $L_{AB}$ 
and $L_{BC}$ are $C^*$-metrics by theorem 6.2 of \cite{R21}.

Motivated by Hanfeng Li's treatment of his nuclear distance \cite{Lih4}, we
consider any amalgamation, $F$, of $D$ and $E$ over $B$. This means
that there are unital injections of $D$ and $E$ into $F$ whose compositions
with the injections of $B$ into $D$ and $E$ coincide. We denote the images 
of $d_0$ and $e_0$ in $F$ again by $d_0$ and $e_0$, and we set 
$f_0 = d_0e_0$. Unfortunately in this generality it could happen that
$f_0 = 0$. (In Hanfeng Li's definition of his nuclear distance this problem 
does not occur since his distinguished elements are, implicitly, the
identity elements.)

\begin{theorem}
\label{thm2.1}
Let notation be as above, and assume that $f_0 \neq 0$. View $F$
as an $A$-$C$-bimodule in the evident way, and define a seminorm,
$N_F$, on $A\oplus C$ by
\[
N_F(a,c) = \|af_0 - f_0c\|_F .
\]
Then for any $\g \geq \g_D + \g_E$ the seminorm $\g^{-1} N_F$ is a bridge for
$(L_A, L_C)$.
\end{theorem}

\begin{proof}
It is clear that
$\g^{-1}N_F(1_A, 0_C) \neq 0$ since $f_0 \neq 0$, and that $\g^{-1}N_F$ 
is norm-continuous. Thus the first two 
conditions of definition 5.1 of \cite{R6} are satisfied. We must verify
the third, final, condition. To simplify notation, we identify $A$, $B$,
$C$, $D$ and $E$ with their images in $F$. For any $a \in A$, $b \in B$
and $c \in C$ we have
\begin{align*}
N_F(a, c) &= \|af_0 - f_0c\|_F  
\leq \|ad_0e_0 - d_0be_0\|_F + \|d_0be_0 - d_0e_0c\|_F \\
&\leq \|ad_0 - d_0b\|_D\|e_0\|_E + \|d_0\|_D\|be_0 - e_0c\|_E \\
&= N_D(a,b) + N_E(b,c) .
\end{align*}
Now let $a \in A$ with $a = a^*$ be given, and let $\e > 0$ be given.
Since $\g_D^{-1}N_D$ is a bridge for $(L_A, L_B)$, there is by definition 
a $b \in B$ with $b^* =b$ such that
\[
L_B(b)\vee \g_D^{-1}N_D(a, b) \leq L_A(a) + \e .
\]
Then since $\g_E^{-1}N_E$ is a bridge for $(L_B, L_C)$, there is a 
$c \in C$ with $c^*=c$ such that 
\[
L_C(c)\vee \g_D^{-1}N_D(b, c) \leq L_B(b) + \e .
\]
Consequently
\[
L_C(c) \leq L_B(b) + \e \leq L_A(a) + 2\e ,
\]
and, from the earlier calculation,
\begin{align*}
N_F(a,c) &\leq N_D(a,b)+N_E(b,c) \leq \g_D(L_A(a) + \e) + \g_E(L_B(b)+\e)  \\
&\leq (\g_D+\g_E)L_A(a) \ + \ \e(\g_D + 2\g_E) .
\end{align*}
The situation is basically symmetric between $A$ and $C$, so one can make 
a similar calculation but starting with a $c \in C$ to obtain a $b \in B$ and
then an $a \in A$ satisfying the corresponding inequalities. This shows that
$(\g_D + \g_E)^{-1}N_F$ is indeed a bridge. Then also $\g^{-1}N_F$ will be
a bridge for any $\g \geq \g_D + \g_E$.
\end{proof}

However, I have so far not seen any good general conditions that yield 
estimates showing that if the corresponding seminorm
\[
L_{AB} = L_A\vee L_B\vee \g^{-1}(N_D\vee N^*_D)
\]
brings $(A,L_A)$ and $(B, L_B)$ close together, and similarly for $L_{BC}$,
then $L_{AC}$ using $(\g_D + \g_E)^{-1}N_F$ brings $(A, L_A)$ and
$(C, L_C)$ close together, in the sense that 
$\mathrm{dist}_H^{\rho_{L_{AC}}}(S(A), S(C))$ is small.
In Hanfeng Li's nuclear distance, in which the distinguished elements
are all, implicitly, the identity elements, this aspect works much better.
And since the nuclear distance satisfies the triangle inequality, it
is clear that $\mathrm{dist}_{nu}(B^m, B^n)$ converges to 0 as
$m$ and $n$ go to $\infty$. But so far I find the nuclear distance to be 
more elusive, as I discuss briefly in section 6 of \cite{R21}, though it is
certainly attractive. I do not yet see how to obtain for the nuclear
distance the kind of quantitative estimates that we will obtain
here for $\mathrm{prox}$.


\section{The proof and statement of the main theorem}
\label{sec3}
For the context of Section 1 the role of $F$ of Section 2 is played by
$C(G/H, B^m\otimes B^n)$, while the roles of $d_0$ and $e_0$ are
played by $\o_m$ and $\o_n$, with $f_0$ being $\o_{mn}$. Let
$\g^A_m$ be defined as in proposition 8.1 of \cite{R21} but for
$P=P^m$, and let $\g^B_m$ be defined as in proposition 8.2
of \cite{R21} but for $P=P^m$. Let $\g_m = \max\{\g^A_m, \g^B_m\}$.
All that we need to know here about $\g_m$ is that propositions 8.1
and 8.2 of \cite{R21} tell us that, for $N_m$ as defined in Section 1 above,
$\g_m^{-1}N_m$ is a bridge for $(L_A, L^B_m)$, and that
propositions 10.1 and 12.1 of \cite{R21} tell us that $\g_m$ 
converges to 0 as $m$ goes to $\infty$. From Theorem 2.1 above
and from the identifications made above, it follows immediately that 
for any $\g$ with $\g \geq \g_m + \g_n$ the seminorm $\g^{-1}N_{mn}$
is a bridge for $(L^B_m, L^B_n)$.

We now investigate how close $S(B^m)$ and $S(B^n)$ are in the
metric from the corresponding seminorm $L_{mn}$ on $B^m\oplus B^n$.
Given $\mu \in S(B^m)$, we want a systematic way to find a
$\nu \in S(B^n)$ that is ``relatively close'' to $\mu$. For this purpose
we use the Berezin symbols $\s^n$ and $\breve \s^n$ that we used
in \cite{R7, R21}. We recall that $\s^n$ is the completely positive unital 
map from $B^n$ to $A$ defined by $\s_T^n(x) = \tr(\a_x(P^n)T)$,
while $\breve \s^n$ is the completely positive unital map from $A$
to $B^n$ defined by
\[
\breve \s^n_f = d_n \int_{G/H} f(x) \a_x(P^n) dx , 
\]
where we recall that $d_n$ is the dimension of $\cH^n$, and the
$G$-invariant measure on $G/H$ gives $G/H$ measure 1. Then
$\breve \s^m \circ \s^n$ will be a completely positive unital map from 
$B^n$ to $B^m$, whose transpose will map $S(B^m)$ into
$S(B^n)$, for any $m$ and $n$. For any $T \in B^n$ we have
\[
\breve \s^m(\s^n_T) = d_m \int_{G/H} \a_x(P^m) \tr(\a_x(P^n)T) dx .
\]

Let $N_{mn}$ be the seminorm on $B^m\oplus B^n$ determined by
$\o_{mn}$, so that 
\begin{align*}
&N_{mn}(S, T) = \|S\o_{mn} - \o_{mn}T\|   \\
&= \sup\{\|(S\otimes I_n)\a_x(P^m\otimes P^n) -  \a_x(P^m\otimes P^n)(I_m\otimes T)\|: x \in G/H\}  .
\end{align*}
Then $L_{mn}$ is defined on $B^m \oplus B^n$ by
\[
L_{mn}(S,T) = L_m^B(S) \vee L_n^B(T)\vee \g^{-1}(N_{mn}(S,T) \vee
N_{mn}(S^*, T^*)) 
\]
for some $\g \geq \g_m + \g_n$.
Let $\mu \in S(B^m)$ be given, and as state $\nu \in S(B^n)$ potentially
close to $\mu$ we choose $\nu$ to be defined by $\nu(T) = \mu(\breve \s^m(\s^n_T))$.
We then want an upper bound on $\rho_{L_{mn}}(\mu, \nu)$.
Now
\[
\rho_{L_{mn}}(\mu, \nu) = \sup\{|\mu(S) - \nu(T)|:L_{mn}(S, T) \leq 1\} ,
\]
and
\[
|\mu(S) - \nu(T)| = |\mu(S) - \mu(\breve \s^m(\s^n_T))| \leq ||S - \breve \s^m(\s^n_T)|| .
\]
So we need to understand what the condition $L_{mn}(S,T) \leq 1$ implies for
$||S - \breve \s^m(\s^n_T)||$. This seems difficult to do directly, so we use a little 
gambit that we have used before, e.g. shortly before notation 8.4 of \cite{R21}, namely
\begin{align*}
||S - \breve \s^m(\s^n_T)|| &\leq ||S - \breve \s^m(\s^m_S)|| + 
|| \breve \s^m(\s^m_S) - \breve \s^m(\s^n_T)||    \\
&\leq \d^B_m L^B_m(S) + \|\s^m_S - \s^n_T\|_\infty  ,
\end{align*}
where for the last inequality we have used theorem 11.5 of \cite{R21}, which includes
the definition of $\d^B_m$. (We remark that theorem 11.5 of \cite{R21} is the
same as theorem 6.1 of \cite{R7}, but \cite{R21} gives a simpler proof of this theorem.)
Note that  $L_{mn}(S,T) \leq 1$ implies that $L^B_m(S) \leq 1$.
Thus we see that it is  $\|\s^m_S - \s^n_T\|_\infty $ that we need to control.
In preparation for this we establish some additional notation in order to
put the situation into a comfortable setting. Notice that 
$B^m\otimes B^n = \cL(\cH^m\otimes \cH^n)$. Furthermore $\xi^m\otimes \xi^n$
is a highest-weight vector in $(U^m\otimes U^n, \cH^m\otimes \cH^n)$, and its weight 
is just the sum of the highest weights of $(U^m\otimes \cH^m)$ and $(U^n\otimes \cH^n)$,
which is just the highest weight of $(U, \cH)$ multiplied by $m+n$. Thus 
$\xi^m\otimes \xi^n$ is just the highest-weight vector for a copy of
$(U^{m+n}, \cH^{m+n})$ inside $\cH^m\otimes \cH^n$.  To simplify notation
we now just set $\xi^{m+n} = \xi^m\otimes \xi^n$, and view $ \cH^{m+n}$
as being the $G$-invariant subspace of $\cH^m\otimes \cH^n$ generated by $\xi^{m+n}$.
Then the rank-1 projection $P^{m+n}$ on $\xi^{m+n}$ is exactly $P^m\otimes P^n$.
We let $\Pi^{mn}$ denote the projection from $\cH^m\otimes \cH^n$ onto $\cH^{m+n}$.
Our notation will not distinguish between viewing the domain of $P^{m+n}$ as being
$\cH^m\otimes \cH^n$ or as being $\cH^{m+n}$, and we will use 
below the fact that $\a_x(P^{m+n} )= \a_x(P^{m+n}) \Pi^{mn}$ 
for any $x \in G$.

\begin{lemma}
\label{lm3.1}
For any $S \in B^m$ and $T \in B^n$ we have
\[
\s^m_S - \s^n_T = \s_R^{m+n}
\]
where $R=\Pi^{mn}(S\otimes I_n \ - \ I_m\otimes T)\Pi^{mn}$, viewed
as an element of $B^{m+n}$.
\end{lemma}

\begin{proof}
For any $x \in G$ we have
\begin{align*}
&\s_S(x) - \s_T(x) = \tr^m(\a_x(P^m)S) - \tr^n(T\a_x(P^n))   \\
&=(\tr^m\otimes \tr^n)(\a_x(P^m\otimes P^n)(S\otimes I_n - I_m\otimes T)
\a_x(P^m\otimes P^n))   \\
&= \tr^{m+n}(\a_x(P^{m+n})\Pi^{mn}(S\otimes I_n - I_m\otimes T)\Pi^{mn})   \\
&= \s_R^{m+n}(x)  .
\end{align*}
\end{proof}

Notice now that for $R$ defined as just above, because the rank of $P^{m+n}$ is 1,
we have for any $x \in G$
\begin{align*}
|\s^{m+n}_R(x)| &= |\tr^{m+n}(\a_x(P^{m+n})\Pi^{mn}(S\otimes I_n - I_m\otimes T)\Pi^{mn})|  \\
&= \|\a_x(P^{m+n})(S\otimes I_n - I_m\otimes T)
\a_x(P^{m+n})\|  \\
&\leq  \|\a_x(P^{m+n})(S\otimes I_n) - (I_m\otimes T)
\a_x(P^{m+n})\|  ,
\end{align*}
and consequently
\[
\|\s^{m+n}_R\| \leq N_{mn}(S^*, T^*) .
\]
But if $L_{mn}(S, T) \leq 1$, then $N_{mn}(S^*, T^*) \leq \g_m + \g_n$ if we
have taken $\g = \g_m + \g_n$. Thus we find that
\[
|\mu(S) - \nu(T)| \leq \d^B_m + \g_m + \g_n .
\]
Since the situation is symmetric in $m$ and $n$, we conclude that
\[
\mathrm{dist}_H^{\rho_{L_{mn}}}(S(B^m), S(B^n)) \leq
\max\{\d^B_m, \d^B_n\} +\max\{\g^A_m, \g^B_m\} + \max\{\g^A_n, \g^B_n\}.
\]
As mentioned in part above, it is shown in proposition 10.1, theorem 11.5, 
and proposition 12.1 of \cite{R21} that, respectively, $\g^A_m$, $\d^B_m$, 
and $\g^B_m$ all converge to 0 as $m$ goes to $\infty$. We thus obtain 
the main theorem of this paper:

\begin{theorem}
With notation as above, for all $m$ and $n$ we have
\[
\mathrm{prox}(B^m, B^n) \leq 
\max\{\d^B_m, \d^B_n\} +\max\{\g^A_m, \g^B_m\} + \max\{\g^A_n, \g^B_n\} ,
\]
and in particular, $\mathrm{prox}(B^m, B^n)$ converges to 0 as $m$ and $n$
go to $\infty$.
\end{theorem}

One can also obtain matricial versions of this theorem along the lines
discussed in section 14 of \cite{R21}.


\def\dbar{\leavevmode\hbox to 0pt{\hskip.2ex \accent"16\hss}d}
\providecommand{\bysame}{\leavevmode\hbox to3em{\hrulefill}\thinspace}
\providecommand{\MR}{\relax\ifhmode\unskip\space\fi MR }
\providecommand{\MRhref}[2]{%
  \href{http://www.ams.org/mathscinet-getitem?mr=#1}{#2}
}
\providecommand{\href}[2]{#2}



\begin{thebibliography}{10}


\bibitem{Krr}
David Kerr, \emph{Matricial quantum {G}romov-{H}ausdorff distance}, 
J. Funct. Anal. \textbf{205} (2003), no.~1, 132--167. 
arXiv:math.OA/0207282. \MR{2020211(2004m:46153)}

\bibitem{KrL}
David Kerr and Hanfeng Li, 
\emph{On {G}romov-{H}ausdorff convergence for operator metric spaces},
J. Operator Theory \textbf{62} (2009), no.~1, 83--109,
arXiv:math.OA/0411157.

\bibitem{Lih2}
Hanfeng Li, \emph{Order-unit quantum {G}romov-{H}ausdorff distance}, 
J. Funct. Anal. \textbf{231} (2006), no.~2, 312--360,
arXiv:math.OA/0312001.
\MR{2195335 (2006k:46119)}
  
\bibitem{Lih3}
\bysame, \emph{C*-algebraic quantum {G}romov-{H}ausdorff distance},
  arXiv:math.OA/0312003.  
  
\bibitem{Lih4}
\bysame, \emph{Metric aspects of noncommutative homogeneous spaces},
J. Funct. Anal.  257 (2009), no.~7, 2325--2350,
arXiv:0810.4694.  
  
\bibitem{R6}
Marc A. Rieffel, \emph{Gromov-{H}ausdorff distance for quantum metric spaces},
Mem. Amer. Math. Soc. \textbf{168} (2004), no.~796, 1--65,
arXiv:math.OA/0011063.
  \MR{2055927}

\bibitem{R7}
\bysame, \emph{Matrix algebras converge to the sphere for quantum
{G}romov-{H}ausdorff distance}, 
Mem. Amer. Math. Soc. \textbf{168} (2004), no.~796, 67--91,
arXiv:math.OA/0108005. \MR{2055928}

\bibitem{R17}
\bysame, \emph{Vector bundles and Gromov-Hausdorff distance}, J. K-Theory, to appear, arXiv:math.MG/0608266.

\bibitem{R21}
\bysame, \emph{Leibniz seminorms for {``M}atrix algebras converge to the sphere{''}},
arXiv:0707.3229.

\bibitem{Wuw3}
Wei Wu, \emph{Quantized {G}romov-{H}ausdorff distance}, 
J. Funct. Anal. \textbf{238} (2006), no.~1, 58--98,
arXiv:math.OA/0503344. \MR{2234123}

\end{thebibliography}

\end{document}